\documentclass[a4paper,10pt]{article}
\usepackage[utf8x]{inputenc}
\usepackage{amsmath,amssymb,amsthm}
\usepackage{verbatim}

\newcommand{\N}{\mathbb{N}}
\newcommand{\Z}{\mathbb{Z}}

\newcommand{\OreGen}{R[x;\sigma,\delta]}

\newcommand{\DiffPol}{R[x;\identity_R,\delta]}

\DeclareMathOperator{\identity}{id}

\makeatletter
\def\imod#1{\allowbreak\mkern10mu({\operator@font mod}\,\,#1)}
\makeatother

\newtheorem{thm}{Theorem}[section]
\newtheorem{remark}[thm]{Remark}

\newtheorem{lemma}[thm]{Lemma}
\newtheorem{prop}[thm]{Proposition}
\newtheorem{cor}[thm]{Corollary}

\title{Centralizers in Ore extensions of polynomial rings}
\author{Johan Richter\footnote{johan.richter@mdh.se} and Sergei Silvestrov\footnote{sergei.silvestrov@mdh.se}}
\date{}

\begin{document}

\maketitle

\begin{abstract}
In this paper we consider centralizers of single elements in certain Ore extensions, with a non-invertible endomorphism, of the ring of polynomials in one variable over a field. We show that they are commutative and finitely generated as an algebra. We also show that for certain classes of elements their centralizer is singly generated as an algebra.  
\\[+2mm]
\end{abstract}

\section{Introduction}

This article is concerned with centralizers of elements in Ore extensions of the form $K[y][x;\sigma,\delta]$, where $K$ is a field, $\sigma$ is an $K$-algebra endomorphism such that $\deg(\sigma(y))>1$ and $\delta$ is a $K$-linear $\sigma$-derivation. 

 We now remind the reader what an Ore extension is. An Ore extension of a ring $R$ is the additive group of polynomials $R[x]$, equipped with a new multiplication, such that $xr=\sigma(r)x+\delta(r)$ for all $r \in R$, for some functions $\sigma$ and $\delta$, on $R$.
This is well-defined if and only if $\sigma$ is an endomorphism and $\delta$ is an additive function such that 
$$\delta(ab)= \sigma(a)\delta(b)+\delta(a)b$$
for all $a$ and $b$ in $R$. We denote the Ore extension by $\OreGen$. The elements, $r \in R$, satisfying $\sigma(r)=r$ and $\delta(r)=0$ are called the \emph{constants} of the Ore extension. In our cases $R$ is an algebra over a field $K$ and we assume that $\sigma$ and $\delta$ are $K$-linear.
See e.g. \cite{GoodearlWarfield} for the definition and basic properties of Ore extensions. 

There is a series of results concerning centralizers in rings of the form $\DiffPol$ in the literature, that has inspired this article. The method of proof we use goes back to an article by Amitsur \cite{Amitsur}, where he proves the following theorem. 

\begin{thm}\label{thm_DiffField}
 Let $K$ be a field of characteristic zero with a derivation $\delta$. Let $F$ denote the subfield of constants. Form the differential operator ring $S=K[x; \identity,\delta]$, and let $P$ be an element of $S$ of 
degree $n>0$. Set $F[P]= \{ \sum_{j=0}^m b_j P^j \ | \ b_j \in F \ \}$, the ring of polynomials in $P$ with constant coefficients. 
Then the centralizer of $P$ is a commutative subring of
$S$ and a free $F[P]$-module of rank at most $n$.  
\end{thm}

Generalizations of this result can be found in an article by Goodearl and Carlson \cite{GoodearlCarlson} and in an article by  Goodearl alone \cite{GoodearlPseudo}. Both articles deal with the case that $\sigma=\identity$, however. Makar-Limanov, in \cite{Makar-Limanov}, studies centralizer in the quantum plane, ie the ring $K[y][x;\sigma, 0]$, with $\sigma(y)=qy$. The results in \cite{Makar-Limanov} also follow from results in \cite{BellSmall}. This article, by Bell and Small, describes centralizers of elements in domains of Gelfand-Kirillov dimension $2$. Some of our results are similar to theirs but are logically independent, since the algebras in this article have infinite Gelfand-Kirillov dimension. 

The paper that comes closest in approach to our paper, that we have been able to find, is an unpublished preprint by Tang \cite{Xin}. Tang also studies Ore extensions over $K[y]$, but with $\sigma$ an automorphism. Like us, Tang describes the structure of maximal commutative subalgebras of the algebras he studies. He cites \cite{Arnal, BavulaWeyl, Dixmier} by Arnal and Pinczon, Bavula respectively Dixmier, as previous articles obtaining similar results on maximal commutative subalgebras. The article by Dixmier contains many results, including similar descriptions of centralizers to the one we give, but it deals exclusively with the Weyl algebra. Bavula's article studies Generalized Weyl algebras and obtains many results, a few of which have analogues in this article. The class of Generalized Weyl algebras does not include our class of Ore extension  however. We have not had access to Arnal's and Pinczon's article, but it appears to deal with a completely different class of algebras from those we study.  

In \cite{ergodipotent}, Hellstr\"om and the second author generalizes Amitsur's method of proof. Among other results, they show that Amitsur's argument works in a large class of graded algebras, provided a condition on the dimension of certain subsets of centralizers is met. We have not found a way to apply their results to the algebras in this article, however. 

This article is a continuation of the article \cite{RichterAGMP}, by the first author. Theorem \ref{thm_Centralizer} can be found in that paper. The arrangement of the proof is somewhat different however. Theorem \ref{thm_Centralizer} complements our other results that describe the centralizer.

In the next section we will introduce some notation and lemmas that we will use throughout this article. In the third section we prove that the centralizer of a non-constant element, $P$, is a free module of finite rank over the ring of polynomials in $P$ with constant coefficients (Theorem \ref{thm_Centralizer}). In the fourth section we prove that centralizers of non-constant elements are commutative (Theorem \ref{thm_cenRcom}) and describe centralizers of any set (Proposition \ref{prop_centralizers}). In the fifth section we try to determine when centralizers are isomorphic to the ring of polynomials in one variable. We manage to prove that this is true in many cases (Propositions \ref{prop_degPrime} and \ref{prop_yLeqx}), with the sufficient conditions given depending only on the leading coefficient. In Propositions \ref{prop_SpecialSigma} and \ref{MonomCentralizer} we restrict the class of Ore extension we are considering and obtain results showing that centralizers of certain elements are isomorphic to the ring of polynomials in one variable.

\section{Preliminaries}

We will adopt the following standing conventions and notations in this article. $K$ is a field and $R=K[y]$ is the polynomial ring in one variable over that field. By $\sigma$ we denote an $K$-algebra endomorphism of $R$ such that $\deg_y(\sigma(y))>1$. By $\delta$ we denote a $\sigma$-derivation on $R$, i.e. a $K$-linear and additive function $R \to R$ such that 
$$\delta(ab)= \sigma(a)\delta(b)+\delta(a)b,$$
 for all $a$ and $b$ in $R$. Our object of study will be the \emph{Ore extension} $S=\OreGen$.  We note that the constants are precisely the elements of $K$. 

We define the notion of the degree of an element in $S$ w.r.t. $x$ in the obvious way. We set $\deg(0):= -\infty$. As for the ordinary degree it is true that $\deg(ab) = \deg(a) +\deg(b)$.  It is important not to confuse this degree function with the degree of an element of $R$ as a polynomial in $y$.  We will always mean degree w.r.t. $x$ when we write degree, unless we explicitly indicate otherwise.

If $A$ is a subset of a ring $B$, then by $C_B(A)$ we denote the \emph{centralizer} of $A$, the set of all elements in $B$ that commute with every element in $A$. If $a$ is a single element we write $C_B(a)$ instead of $C_B(\{a\})$.  

We start with two lemmas that will be important in what follows.

\begin{lemma}\label{onedim}
 Suppose that $P$ is any element in $S$, and that $Q \in C_S(P)$ has degree $m$. Let $q_m$ be the leading coefficient of $Q$. Then 
\begin{equation}\label{eq_highCoeff} 
 p_n \sigma^n(q_m) = q_m \sigma^m(p_n).
\end{equation}
 The solution space of this equation (as an equation for $q_m$) is at most one-dimensional as a $K$-sub vector space of $K[y]$. 
\end{lemma}

\begin{proof}

The equation follows by equating the highest order coefficients in $PQ$ and $QP$. To show that the solution space is one-dimensional we begin by noting that if $\rho=\deg_y(p_n)$, $s=\deg_y(\sigma(y))$ and $k=\deg_y(q_m)$ then 
\begin{equation*}
\rho + s^n k = k+ s^m \rho.
\end{equation*}
Thus $k$ is determined uniquely. Now suppose that $a,b$ are two solutions of Equation (\ref{eq_highCoeff}). Then we can find $\alpha \in K$, such that $\deg_y(a-\alpha b) < k$. But since $a-\alpha b$ is another solution of \eqref{eq_highCoeff} it follows that $a=\alpha b$.

\end{proof}

\begin{lemma}
 For any $P \in S$ of degree larger than $0$ it is true that 
\begin{displaymath}
 C_S(P) \cap R = K.
\end{displaymath} 
 \end{lemma}

\begin{proof}
This follows from Lemma \ref{onedim}.
\end{proof}

\section{Centralizers are free $K[P]$-modules}

\begin{thm}\label{thm_Centralizer}
 Let $P$ be any element of $S$ that is not constant. Then $C_S(P)$ is a free $K[P]$-module of rank at most $n:=\deg(P)$. 
     
\end{thm}    

The proof we give is similar to one in \cite{Amitsur}. As noted above, the theorem can also be found in \cite{RichterAGMP}.

\begin{proof} 

Denote by $M$ the subset of elements of $\{0, 1, \ldots, n-1 \}$ such that an integer $0 \leq i < n$ is in $M$ if and only if $C_S(P)$ contains an element of degree equivalent to $i$ modulo $n$. For $i \in M$ let $p_i$ be an element in $C_S(P)$ such that 
$\deg(p_i) \equiv i   \imod n$ and $p_i$ has minimal degree for this property. Take $p_0=1$. 

We will show that $\{ p_i | i \in M \}$ is a basis for $C_S(P)$ as a $K[P]$-module.

We start by showing that the $p_i$ are linearly independent over $K[P]$. Suppose $\sum_{i \in M} f_i p_i =0$ for some $f_i \in K[P]$. If $f_i \neq 0$, for a particular $i$, then $\deg(f_i)$ is divisible by $n$, in which case 
\begin{align}
 \deg(f_i p_i)= \deg(f_i)+\deg(p_i) \equiv \deg(p_i) \equiv i  \imod n. 
\end{align}
If $\sum_{i \in M} f_i p_i=0$ but not all $f_i$ are zero, we must have two nonzero terms, $f_i p_i$ and $f_j p_j$, that have the same degree despite $i,j \in M$ being distinct. But this is impossible since $ i \not\equiv j 
\imod n$. 

We now proceed to show that the $p_i$ span $C_S(P)$. Let $W$ denote the submodule they do span. We see that $W$ contains all elements of degree $0$ in $C_S(P)$. 

Now assume that $W$ contains all elements in $C_S(P)$ of degree less than $j$. Let $Q$ be an element in $C_S(P)$ of degree $j$. There is some $i$ in $M$ such that $j \equiv i \imod n$. Let $m$ be the degree of $p_i$. By the 
choice of $p_i$ we now that $m \equiv j \imod n$ and $m \leq j$. Thus $j= m+qn$ for some non-negative integer $q$. The element $P^q p_i$ lies in $W$ and has degree $j$. By Lemma \ref{onedim} the leading coefficient of $Q$ equals 
the leading coefficient of $P^q p_i$ times some constant $\alpha$. The element $Q-\alpha P^q p_i$  then lies in $C_S(P)$ and has degree less than $j$. By the induction hypothesis it also lies in $W$, and hence so does $Q$. 
\end{proof}

\section{Centralizers are commutative}

We now prove that the centralizer of any non-constant element of $S$ is commutative. For the proof of this we once again follow closely the presentation in \cite{Amitsur}. 

\begin{thm}\label{thm_cenRcom}
 Let $P$ be an element of $S$ that is not a constant. Then $C_S(P)$ is commutative.
\end{thm}

\begin{proof}
 If $P$ is an element of $R \setminus K$ it follows that $C_S(P)=R$ which is commutative. Thus suppose that $n=\deg(P) \geq 1$. Let $D$ be the set of degrees of non-zero elements of $C_S(P)$. Since $C_S(P)$ is a subring, and 
$\deg(ab) =\deg(a)+\deg(b)$ for any non-zero $a,b$, it follows that $D$ is closed under addition. Map $D$ into $\Z_n$ in the natural way and denote the image by $D_n$. Since $D_n$ is finite, closed under addition and contains $0$ it 
is a subgroup of $Z_n$. So it is a cyclic group. 

Let $Q \in C_S(P)$ be an element such that $ \deg(Q) \mod n$ generates $D_n$. Let $J$ be the set of elements of the form
\begin{displaymath}
 H(P,Q) = \phi_0 + \phi_1 Q + \ldots \phi_l Q^l, \ \phi_i \in K[P], \ i=0\ldots l
\end{displaymath}
and let $E = \{ \deg(H(P,Q)) \ | \ H(P,Q) \in J \}$.  
Suppose that $t \in \N$ is such that if $m\geq t$ and $m \in D$ then $m \in E$. Such a $t$ must clearly exist. Suppose now that  $U$ is any element of $C_S(P)$. If $\deg(U) \geq t$, then, by Lemma \ref{onedim}, there is a $H_1(P,Q) \in J$
such that $\deg(U-H_1) < \deg(U)$. By repeating this process if necessary, we find that we can write $U =H(P,Q) + U_0$ where $\deg(U_0) < t$.  We note that the set of elements in $C_S(P)$ of degree less than $t$ form a 
finite-dimensional vector space over $K$ of dimension at most $t$. 

If $V$ is an element of $C_S(P)$ we can write $VP^i = H_i(P,Q)+ V_i$, where $\deg(V_i) < t$, for $i=0,1,\ldots t$. Then the $V_i$ are linearly dependent so there are $c_i \in K$ such that 
$\sum_{i=0}^t c_i V_i =0$ which implies that 
\begin{displaymath}
V \sum_{i=0}^t c_i P^i = \sum_{i=0}^t c_i H_i.  
\end{displaymath}
So for any $V \in C_S(P)$, there are non-zero $f \in K[P]$ and $H(P,Q) \in J$ such that $V f(P) = H(P,Q)$. The elements in $J$ commute with each other and the elements of $K[P]$ commutes with everything in $C_S(P)$. Thus if
$V_1, V_2$ are two elements in $C_S(P)$, with $V_i f_i(P) = H_i(P,Q)$, we get that 
\begin{multline}
 V_1 V_2 f_1(P) f_2(P) = V_1 f_1(P) V_2 f_2(P) = H_1(P,Q) H_2(P,Q) = \\
=  H_2(P,Q) H_1(P,Q) = V_2 f_2(P) V_1 f_1(P) = V_2 V_1 f_1(P) f_2(P).
\end{multline}

Since $S$ is a domain this implies that $V_1 V_2 =V_2 V_1$.

\end{proof}

It is clear that if $A$ is any set containing a non-constant element then $C_S(A)$ is commutative as well. But we can say more than that, as the next proposition illustrates.

\begin{prop}\label{prop_centralizers}

Let $A$ be any subset of $S$. Then $C_S(A)$ equals either $S$, $K$ or $C_S(P)$, where $P$ is a non-constant element in $S$.
\end{prop}

\begin{proof}
Suppose $A$ contains two elements $P$ and $Q$ (necessarily non-constant), that do not commute with each other. Then 
$$C_S(A) \subseteq C_S(P) \cap C_S(Q).$$

But if $U$ is some non-constant element in $C_S(P) \cap C_S(Q)$, then $P,Q \in C_S(U)$ and by Theorem \ref{thm_cenRcom} it 
would follow that $P$ and $Q$ commute. Thus $C_S(A) =K$. 

Now suppose $A$ contains a non-constant $P$ and everything in $A$ commutes with $P$. Clearly $C_S(A) \subseteq C_S(P)$. But, since $C_S(P)$ is commutative and $A \subseteq C_S(P)$, every element in $C_S(P)$ commutes with every element in $A$. Thus $C_S(P) = C_S(A)$. 

If, finally, $A$ contains only constants, then $C_S(A)=S$.  
\end{proof}

\begin{remark}
We note that the maximal commutative subrings of $S$ are $K$ and the sets of the form $C_S(P)$, for nonconstant $P$. 

\end{remark}

\section{Singly generated centralizers}

We note that we can give a bound on the number of generators needed to generate a centralizer as an algebra. 

\begin{cor}\label{cor_FinGen}
Let $P \in \OreGen$ satisfy $n= \deg(P) >0$. Then we can find $n$ elements that generate $C_S(P)$ as a $K$-algebra.
\end{cor}

\begin{proof}
Follows from Theorem \ref{thm_Centralizer} and its proof.
\end{proof}

In some cases we have been able to prove that the centralizer of an element is in fact generated by a single element, not just a finite number of them. To do so we have relied on the the equation stated in Lemma \ref{onedim}.

We begin with a lemma which we will use frequently. 

\begin{lemma} \label{lem_degDiv}

Let $P$ be a non-constant element of $S$ of degree $n$. Suppose all elements of $C_S(P)$ have degree divisible by $n$. Then 
\begin{displaymath}
C_S(P)=K[P] := \{ \,  \sum c_i P^i \mid c_i \in K \}.
\end{displaymath}
\end{lemma}

\begin{proof}
We know that $ K[P] \subseteq  C_S(P)$. We also know that all elements of degree zero in $C_S(P)$ lie in $K[P]$. We give a proof by induction. 

Suppose that elements in $C_S(P)$ of degree less than $k$ lie in $K[P]$. We want to show that all elements of degree $k$ in $C_S(P)$ also lie in $K[P]$. If $k$ is not divisible by $n$ this is vacuously true. So suppose $k=pn$ for some integer $p$ and let $Q$ be any element in $C_S(P)$ of degree $k$. The element $P^p$ lies in $C_S(P)$ and has degree $k$. By Lemma \ref{onedim} there is an $\alpha \in K$ such that $Q$ and $\alpha P^p$ have the same leading coefficient. Thus we have that $\deg(Q -\alpha P^p) < k$ which implies that $Q-\alpha P^p \in K[P]$, by the induction assumption. Hence it follows that $Q \in K[P]$.  
\end{proof}

Our first result showing a centralizer to be singly generated is in the case when our non-constant element has prime degree. 

\begin{prop}\label{prop_degPrime}

Let $P$ be an element of $S$ of degree $n$, where $n$ is a prime. Let $p_n$ be the leading coefficient of $P$ and let $\rho$ be the degree of $p_n$ as a polynomial in $y$. Let $s$ be the degree of $\sigma(y)$, also as a 
polynomial in $y$. 
Then if $\sum_{i=0}^{n-1} s^i$ does not divide $\rho$ it follows that $C_S(P)= \{ \sum c_i P^i \ | \ c_i \in K \}$.
\end{prop}

We use the following lemma in our proof.

\begin{lemma}\label{lem_geomSum}
Let $m$ and $n$ be positive integers and suppose that $\gcd(n,m)=1$. Let $s$ be a positive integer. Then $\gcd(\sum_{i=0}^{n-1} s^i, \sum_{j=0}^{m-1} s^j )=1$. 
\end{lemma}

\begin{proof} 
This is clearly (vacuously) true for $n=1$ and it is a simple exercise to prove it is true for $n=2$. We use induction on $n$ to prove the lemma in general. So suppose it is true if $n<k$ and we want to show it is true for $n=k$. 
So let $m>k$ be such that $\gcd(k,m)=1$.

\begin{multline*}
\gcd(\sum_{i=0}^{k-1} s^i, \sum_{j=0}^{m-1} s^j ) = \gcd\left(\sum_{i=0}^{k-1} s^i, \sum_{j=0}^{k-1} s^j + \sum_{j=k}^{m-1}s^j\right) = \gcd\left( \sum_{i=0}^{k-1} s^i, s^k \sum_{j=0}^{m-1-k} s^j\right) = \\
= \gcd\left( \sum_{i=0}^{k-1} s^i, \sum_{j=0}^{m-1-k} s^j\right).
\end{multline*}
 Now it is clearly true that $\gcd(k, m-k)=1$. If $m-k<k$ we can use the induction assumption. If  $m-k>k$ set $m'= m-k$ and repeat the previous calculation. Sooner or later we will reduce to a case where we can use the induction 
assumption. 
 
\end{proof}

\begin{proof}[Proof of proposition]
Let $Q$ be an element of $S$ that commutes with $P$. Let $Q$ have degree $m$ and suppose that $\gcd(m,n)=1$. Let $q_m$ be the leading coefficient of $Q$. Equating the leading coefficients in $PQ$ and $QP$ we
find that 
\begin{displaymath}
 p_n \sigma^n(q_m) = q_m \sigma^m(p_n). 
\end{displaymath}
 
If $k$ denotes the degree of $q_m$, we find that 
\begin{displaymath}
 k= \rho \frac{s^m-1}{s^n-1} = \rho \frac{\sum_{i=0}^{m-1}s^i}{\sum_{i=0}^{n-1} s^i}.
\end{displaymath}
Now it would follow from the lemma that $k$ is a non-integer which is impossible. 

Thus $\gcd(m,n)=n$, since $n$ is prime, and the result follows by Lemma \ref{lem_degDiv}.

\end{proof}

We can generalize Lemma \ref{lem_geomSum} to the following lemma.

\begin{lemma}\label{lem_gcd}
 Let $m$ and $n$ be positive integers. Let $s$ be a positive integer greater than $1$. Set $r=\gcd(m,n)$. Then 
\begin{displaymath}
\gcd\left(\sum_{i=0}^{m-1} s^i, \sum_{j=0}^{n-1} s^j\right) = \sum_{i=0}^{r-1} s^i. 
\end{displaymath}

\begin{proof} 
We can write
\begin{align*}
\sum_{i=0}^{m-1} s^i = \sum_{k=0}^{\frac{m}{r}-1} s^{rk} \sum_{i=0}^{r-1} s^i = \left( \sum_{k=0}^{\frac{m}{r}-1} s^{rk} \right) \left( \sum_{i=0}^{r-1} s^i \right) 
\end{align*}
 
and 
\begin{align*}
\sum_{i=0}^{n-1} s^i = \sum_{k=0}^{\frac{n}{r}-1} s^{rk} \sum_{i=0}^{r-1} s^i = \left( \sum_{k=0}^{\frac{n}{r}-1} s^{rk}\right) \left( \sum_{i=0}^{r-1} s^i \right) .
\end{align*}

Since $\gcd\left(\sum_{k=0}^{\frac{m}{r}-1} s^{rk}, \sum_{k=0}^{\frac{n}{r}-1} s^{rk}\right)=1$ by Lemma \ref{lem_geomSum} we are done. 
\end{proof}

\end{lemma}

We use this lemma in the next proposition. 

\begin{prop}\label{prop_yLeqx}
 Let $P$ be an element of $S$ of degree $n>0$ in $x$ and suppose that $p_n$ (the leading coefficient of $P$) has degree greater than zero but not greater than $n$ as a polynomial in $y$. Then $C_S(P)=K[P]$. 
\end{prop}

\begin{proof}
 When $n=1$ this is true by Corollary \ref{cor_FinGen}. When $n=2$ or $n=3$ this is true by Proposition \ref{prop_degPrime}. So suppose that $n \geq 4$. 

It will be enough to prove that the degrees of all elements of $C_S(P)$ are divisible by $n$ by Lemma \ref{lem_degDiv}.

Let $Q$ be an element of $C_S(P)$. Suppose that $Q$ has degree $m$. Let $q_m$ be the leading coefficient of $Q$. By comparing the leading coefficient of $PQ$ and $QP$ we get the equation
\begin{displaymath}
  p_n \sigma^n(q_m) = q_m \sigma^m(p_n).
\end{displaymath}

Let $k$ denote the degree of $q_m$ and $\rho$ the degree of $p_n$. (Both degrees are measured as polynomials in $y$.) We get the following equation for $k$. 
\begin{displaymath}
 k = \rho \frac{\sum_{i=0}^{m-1}s^i}{\sum_{i=0}^{n-1}s^i}.
\end{displaymath}

Set $r=\gcd(m,n)$. What we want to prove is that $r=n$. So suppose that it does not equal $n$. Then $r\leq \frac{n}{2}$. Write $n=r n'$. Then 
\begin{displaymath}
 \sum_{i=0}^{n-1} s^i = \left( \sum_{i=0}^{r-1} s^i \right) \left( \sum_{i=0}^{n'-1} s^{ri} \right).
\end{displaymath}

From Lemma \ref{lem_gcd} we conclude that $\sum_{i=0}^{n'-1} s^{ri}$ must divide $\rho$ if $k$ is to be an integer. However, 
\begin{displaymath}
 \sum_{i=0}^{n'-1} s^{ri} > s^{r(n'-1)} \geq 2^{r(n'-1)} = \frac{2^n}{2^r}.
\end{displaymath}
  
Since $r \leq \frac{n}{2}$ we find that
\begin{displaymath}
 \frac{2^n}{2^r} \geq 2^{\frac n 2}.
\end{displaymath}

Since $2^{\frac n 2} \geq n$ for all $n \geq 4$ we find, to summarize our calculations, that 
\begin{displaymath}
 \sum_{i=0}^{n'-1} s^{ri} > 2^{\frac n 2} \geq n \geq \rho.
\end{displaymath}

But this is a contradiction to the fact that the sum had to divide $\rho$. 

\end{proof}

\begin{cor}
 Let $n$ be any positive integer. Then $C_S(y^n x^n) = K[y^n x^n]$. 
\end{cor}

\begin{prop}
 Let $n$ be any positive integer. Then $C_S(x^n y^n) = K[x^n y^n]$. 
\end{prop}

\begin{proof}
 Set $P = x^n y^n$. $P$ has degree $n$ as an element of $S$ and its leading coefficient is $\sigma^n(y^n)$. The degree of the leading coefficient as a polynomial in $y$ is $n s^n$. 

The proposition is true when $n=1$ by Corollary \ref{cor_FinGen}. It is true when $n=2$ and when $n=3$ by Proposition \ref{prop_degPrime}. 

So suppose that $n \geq 4$. Let $Q$  be an element of degree $m$. As before it suffices to prove that $\gcd(m,n)=n$. We will use a proof by contradiction, so set $r=\gcd(m,n)$ and suppose that $r<n$.  Letting $k$ denote 
the degree in $y$ of the leading coefficient of $Q$ we get, 
as before,
\begin{displaymath}
 k = n s^n \frac{\sum_{i=0}^{m-1}s^i}{\sum_{i=0}^{n-1}s^i}.
\end{displaymath}
 
We cancel common factors in the fraction, and by Lemma \ref{lem_gcd} we get
\begin{displaymath}
 k = n s^n \frac{A}{\sum_{i=0}^{n'-1} s^{ri}},
\end{displaymath}
where $n' = \frac n r$.  Since $\gcd(A s^n, \sum_{i=0}^{n'-1} s^{ri}) = 1$ we see that we must have that $\sum_{i=0}^{n'-1} s^{ri}|n$. But, as in the proof of Proposition \ref{prop_yLeqx}, this is not the case. 

\end{proof}

For the next proposition we consider only special $\sigma$. 

\begin{prop}\label{prop_SpecialSigma}
Let $R=K[y]$ and suppose that $\sigma(y)=y^k$ for some positive integer $k>1$. Let $P$ be an element of $S=\OreGen$ of degree $n$ and let $p_n$ be its leading coefficient. Suppose that $p_n$ has the following property: 
there does not exist an $a \in \bar{K}$ and distinct positive integers $i,j$, such that $a^i$ and $a^j$ both are roots of $p_n$. (Here $\bar{K}$ is the algebraic closure of $K$.) Then $C_S(P)=K[P]$. 

\end{prop}

\begin{proof}
 Let $Q$ be an element of $C_S(P)$. As before it suffices to prove that $\deg(Q)$ is divisible by $n$. So suppose $m=\deg(Q)$ is not. Let $q_m$ be the leading coefficient of $Q$. We get the following equation
\begin{displaymath}
 p_n \sigma^n(q_m) = q_m \sigma^m(p_n).
\end{displaymath}

Due to the special form of $\sigma$ this can be written 
\begin{displaymath}
 p_n(y) q_m(y^{k^n}) = q_m(y) p_n(y^{k^m}).
\end{displaymath}
Consider $\gcd(p_n(y), p_n(y^{k^m}))$. If this equals a nonzero polynomial $h$, then $h$ has a root, $a$, in $\bar{K}$. But then both $a$ and $a^{k^m}$ would be roots of $p_n$, contradicting the assumption we made. Thus 
$\gcd(p_n(y), p_n(y^{k^m}))=1$. So $p_n(y)$ must divide $q_m(y)$. Set $q_m(y) = p_n(y) \hat{q}(y)$ and simplify.  

The simplified equation becomes
\begin{displaymath}
 p_n(y^{k^n}) \hat{q}(y^{k^n}) = \hat{q}(y) p_n(y^{k^m}).
\end{displaymath}
Now we have that $\gcd(p_n(y^{k^n}),p_n(y^{k^m}))=1$. Thus $\hat{q}=q'(y) p_n(y^{k^n})$ for some $q'$. Inserting this into our equation and simplifying we get 
\begin{displaymath}
 p_n(y^{k^{2 n}}) q'(y^{k^n}) = q'(y) p_n(y^{k^m}).
\end{displaymath}
Since $n$ does not divide $m$ we must have that $2n \neq  m$. Thus 
\begin{displaymath}
\gcd(p_n(y^{k^{2 n}}), p_n(y^{k^{m}})) =1.
\end{displaymath}
 We trust that the pattern is obvious now. It is clear that we can continue this process 
for ever and conclude that $q_m(y)$ is divisible by an infinite sequence of polynomials with strictly increasing degrees. Thus our assumption that $m$ was not divisible by $n$ leads to a contradiction. 

\end{proof}

Specialising the definition of $S$ even further we get the following proposition.

\begin{prop}
\label{MonomCentralizer}
Let $\sigma(y)=y^s$ and $\delta(y)=0$. Set $P =y^i x^j$, where $i+j>0$. Then $C_S(P)$ is singly generated.

\end{prop}

\begin{proof}

The result is clear when $j=0, i>0$ so suppose that $j>0$. 

Suppose that $Q$ belongs to $C_S(P)$. Write $Q= \sum a_{l,k} y^l x^k$. We can compute that 
$$ y^{i}x^{j}y^l x^k = y^{i+l s^j} x^{j+k}.$$

Since $C_S(P)$ is graded by the powers of $x$ it follows that $\sum_l a_{l,k} y^l x^k \in C_S(P)$ for every $k$. Since the 
product of monomials is a new monomial it follows, by induction downwards on the degree in $y$, that every term $a_{l,k} y^l x^k$ must commute with $P$. 

Suppose $a_{l,k} \neq 0$. Then we must have that $$ y^i x^j y^l x^k = y^l x^k y^i x^j,$$ which implies that 
$i+l\cdot s^j = l+i\cdot s^k$. This means that 
$$ l = \frac{s^k-1}{s^j-1} i.$$ We can write this as 
$$ l = \frac{\sum_{m=0}^{k-1} s^m}{\sum_{m=0}^{j-1} s^m} i.$$

For every choice of $i,j,s,k$ this determines $l$. However the formula might give non-integer values for $l$, which does not correspond to an element of $S$. Let $k_0$ be the least non-negative integer for which the RHS is an integer when we substitute $k_0$ for $k$. 

Let $k_1$ be the next least non-negative integer such that the RHS of the formula is an integer.  We compute
$$ \frac{\sum_{m=0}^{k_1-1} s^m}{\sum_{m=0}^{j-1} s^m} i = \frac{\sum_{m=0}^{k_0-1} s^m+ s^{k_0}\sum_{m=0}^{k_1-k_0-1}s^m} {\sum_{m=0}^{j-1} s^m} i. $$ 

We see that (by the definition of $k_0$ and since $\gcd(s^{k_0},\sum_{m=0}^{j-1} s^m)=1$) that 
$$ \frac{\sum_{m=0}^{k_1-k_0-1} s^m}{\sum_{m=0}^{j-1} s^m} i$$
is an integer. This implies that $k_1 = 2k_0$, by the definition of $k_0$. Similarly, all $k$ that give an integer value for $l$ must be multiples of $k_0$. The result is now clear. 
\end{proof}

Note that the proof of Proposition \ref{MonomCentralizer} establishes that the generator of $C_S(P)$ is $y^l x^k$ where $k$ is the least non-negative integer such that 
$$ \frac{\sum_{m=0}^{k-1} s^m}{\sum_{m=0}^{j-1} s^m} i$$ 
is an integer and $l$ is the value of that integer. 

\section*{Acknowledgement} 

This research was supported in part by The Swedish Foundation for International Cooperation in Research and Higher Education
(STINT), The Swedish Research Council, The Swedish Royal Academy of Sciences and The Crafoord Foundation.

The authors also wish to thank an anonymous referee for comments that improved the paper.

\end{document}